\documentclass{amsart}

\usepackage{amsmath,amssymb,amsthm}
\usepackage{url,verbatim}
\usepackage{enumerate}

\usepackage[T1]{fontenc}
\usepackage[utf8]{inputenc}

\newtheorem{theorem}{Theorem}
\newtheorem{lemma}[theorem]{Lemma}
\newtheorem{corollary}[theorem]{Corollary}
\newtheorem{proposition}[theorem]{Proposition}

\newtheorem{Claim}[theorem]{Claim}
\newtheorem{fact}[theorem]{Fact}

\theoremstyle{definition}
\newtheorem{definition}[theorem]{Definition}

\newtheorem{remark}[theorem]{Remark}
\newtheorem{question}[theorem]{Question}
\newtheorem{problem}[theorem]{Problem}

\newcommand{\cM}{\mathcal{M}}
\def\Z{\mathbb Z}
\def\Th{\operatorname{Th}}
\def\NIP{\operatorname{NIP}}
\newcommand{\seq}{\subseteq}
\def\R{\mathbb R}
\newcommand{\Zbar}{\mathbb{Z}^*}
\def\C{\mathbb C}
\def\Q{\mathbb Q}
\newcommand{\cL}{\mathcal{L}}
\def\N{\mathbb N}
\newcommand{\nv}{\text{-}}
\newcommand{\inv}{^{\nv 1}}
\newcommand{\smd}{\raisebox{.75pt}{\textrm{\scriptsize{~\!$\triangle$\!~}}}}
\newcommand{\miff}{\makebox[.4in]{$\Leftrightarrow$}}
\newcommand{\abar}{\bar{a}}

\newcommand{\cbar}{\bar{c}}
\newcommand{\xbar}{\bar{x}}
\newcommand{\ybar}{\bar{y}}
\newcommand{\zbar}{\bar{z}}
\def\rk{\operatorname{rk}}
\newcommand{\uth}{^{\textrm{th}}}
\newcommand{\fbar}{\bar{f}}
\newcommand{\gbar}{\bar{g}}
\newcommand{\ubar}{\bar{u}}

\newcommand{\noit}[1]{\noindent\textit{#1}}
\newcommand{\claim}{\hfill$\dashv_{\text{\scriptsize{claim}}}$}
\newcommand{\jbar}{\bar{\jmath}}
\newcommand{\mand}{\makebox[.4in]{and}}
\newcommand{\mbar}{\bar{m}}
\def\dom{\operatorname{dom}}
\newcommand{\cP}{\mathcal{P}}
\newcommand{\cI}{\mathcal{I}}

\newbox\usefulbox

\makeatletter
\def\getslant #1{\strip@pt\fontdimen1 #1}

\def\skov #1{\mathchoice
 {{\setbox\usefulbox=\hbox{$\m@th\displaystyle #1$}%
    \dimen@ \getslant\the\textfont\symletters \ht\usefulbox
    \divide\dimen@ \tw@ 
    \kern\dimen@ 
    \overline{\kern-\dimen@ \box\usefulbox\kern\dimen@ }\kern-\dimen@ }}
 {{\setbox\usefulbox=\hbox{$\m@th\textstyle #1$}%
    \dimen@ \getslant\the\textfont\symletters \ht\usefulbox
    \divide\dimen@ \tw@ 
    \kern\dimen@ 
    \overline{\kern-\dimen@ \box\usefulbox\kern\dimen@ }\kern-\dimen@ }}
 {{\setbox\usefulbox=\hbox{$\m@th\scriptstyle #1$}%
    \dimen@ \getslant\the\scriptfont\symletters \ht\usefulbox
    \divide\dimen@ \tw@ 
    \kern\dimen@ 
    \overline{\kern-\dimen@ \box\usefulbox\kern\dimen@ }\kern-\dimen@ }}
 {{\setbox\usefulbox=\hbox{$\m@th\scriptscriptstyle #1$}%
    \dimen@ \getslant\the\scriptscriptfont\symletters \ht\usefulbox
    \divide\dimen@ \tw@ 
    \kern\dimen@ 
    \overline{\kern-\dimen@ \box\usefulbox\kern\dimen@ }\kern-\dimen@ }}%
 {}}
\makeatother

\title[Intermediate structures between $(\Z,+,0)$ and $(\Z,+,<,0)$]{There are no intermediate structures between the group of integers and Presburger arithmetic}

\author{Gabriel Conant}

\address{Department of Mathematics\\
University of Notre Dame\\
Notre Dame, IN, 46656, USA}

\date{July 18, 2018}

\begin{document}

\begin{abstract}
We show that if a first-order structure $\cM$, with universe $\Z$, is an expansion of $(\Z,+,0)$ and a reduct of $(\Z,+,<,0)$, then $\cM$ must be interdefinable with $(\Z,+,0)$ or $(\Z,+,<,0)$.
\end{abstract}

\maketitle

\section{Introduction}

\setcounter{theorem}{0}
\numberwithin{theorem}{section}

Suppose $\cM_1$ and $\cM_2$ are first-order structures, with the same underlying universe $M$, but possibly in different languages. We say $\cM_1$ is a \emph{reduct} of $\cM_2$ (equivalently, $\cM_2$ is an \emph{expansion} of $\cM_1$) if, for all $n>0$, every subset of $M^n$ definable in $\cM_1$ is definable in $\cM_2$, where definability always means \emph{with parameters from $M$}. We say $\cM_1$ and $\cM_2$ are \emph{interdefinable} if each structure is a reduct of the other. The following is our main result.
 
 \begin{theorem}\label{thm:premain}
 Suppose $\cM$ is an expansion of $(\Z,+,0)$ and a reduct of $(\Z,+,<,0)$. Then $\cM$ is interdefinable with either $(\Z,+,0)$ or $(\Z,+,<,0)$.
 \end{theorem}

 The motivation for this result is the growing interest in ``tame'' expansions of the group $(\Z,+,0)$, which can be seen as a discrete analog to the prolific study of tame expansions of the real field. Since $(\Z,+,0)$ is a stable $U$-rank $1$ group, there are a wide variety of model theoretic properties one can use as a notion of tameness (see, e.g., \cite{ADHMS2}, \cite{ADHMS1}, \cite{CoPi}, \cite{DoGo}, \cite{KaSh}, \cite{PaSk}).  The following questions, which more specifically motivate Theorem \ref{thm:premain}, involve tameness arising from \emph{stability} and \emph{dp-minimality}. 

\begin{question}\label{ques:two}$~$
\begin{enumerate}[$(a)$]
\item (Marker 2011) Is there a proper expansion $\cM$ of $(\Z,+,0)$ such that $\cM$ is a reduct of $(\Z,+,<,0)$ and $\Th(\cM)$ is stable?
\item (Aschenbrenner, Dolich, Haskell, Macpherson, Starchenko \cite{ADHMS2} 2013) Is every dp-minimal expansion of $(\Z,+,0)$ a reduct of $(\Z,+,<,0)$?
\end{enumerate}
\end{question}

 From a more general perspective, these questions can be viewed as initial incursions into the following ambitious projects.

\begin{problem}\label{prob}$~$
\begin{enumerate}[$(a)$]
\item Classify the stable expansions of $(\Z,+,0)$.
\item Classify the dp-minimal (i.e. dp-rank $1$) expansions of $(\Z,+,0)$.
\end{enumerate}
\end{problem}

We refer the reader to \cite{PiGST} for background on stability, and \cite{DGL}, \cite{Sibook} for background on dp-rank in $\NIP$ theories. While these notions provide motivation for our work, the proof of Theorem \ref{thm:premain} does not require any familiarity with them.

The main reason that \emph{Presburger arithmetic} $(\Z,+,<,0)$ arises in both Questions \ref{ques:two}$(a)$ and \ref{ques:two}$(b)$ is that it is a canonical example of a ``tame'' proper expansion of $(\Z,+,0)$ (versus a ``wild'' expansion like the ring $(\Z,+,\cdot,0,1)$).  For instance, $(\Z,+,<,0)$ is \emph{quasi-o-minimal}, and thus dp-minimal (see \cite{BPW}, \cite{DGL}). Moreover, definable sets in $(\Z,+,<,0)$ enjoy a cell decomposition similar to those in o-minimal structures (this result is due to Cluckers \cite{Cluck}; we give a brief summary in Section \ref{sec:CD}).

At the time Marker posed Question \ref{ques:two}$(a)$, there were no known examples of proper stable expansions of $(\Z,+,0)$. Given the known tameness of $(\Z,+,<,0)$, and well-understood structure of Presburger definable sets, it became natural to ask if such sets could produce proper stable expansions of $(\Z,+,0)$. However, as $(\Z,+,<,0)$ is unstable, Theorem \ref{thm:premain} immediately implies a negative answer to Question \ref{ques:two}$(a)$. We should remark that proper stable expansions of $(\Z,+,0)$ are now known to exist (due to Palac\'{i}n and Sklinos \cite{PaSk} and Poizat \cite{PoZ}, independently). However, in recent joint work with Pillay \cite{CoPi}, we show there are no proper stable expansions of $(\Z,+,0)$ of finite dp-rank. This provides an alternative proof of the answer to Question \ref{ques:two}$(a)$ which, unlike this paper, relies on technology from stable group theory.

Motivation for Question \ref{ques:two}$(b)$ comes from several negative results on tameness in expansions of $(\Z,+,<,0)$. Specifically, while proper $\NIP$ expansions of $(\Z,+,<,0)$ exist (e.g. by \cite{PointPA}), there are many results showing that such expansions do not satisfy various refinements of $\NIP$. For example, Belegradek, Peterzil, and Wagner \cite{BPW} showed in 2000 that $(\Z,+,<,0)$ has no proper quasi-o-minimal expansions. This was generalized in 2013 to dp-minimal expansions by Aschenbrenner, Dolich, Haskell, Macpherson, and Starchenko \cite{ADHMS1}, and finally in 2016 to \emph{strong} (which includes finite dp-rank) expansions by Dolich and Goodrick \cite{DoGo}. Each of these results relies on a powerful fact, due to Michaux and Villemaire \cite{MiVi}, that any \emph{proper} expansion of $(\Z,+,<,0)$ defines a \emph{unary} set $A\seq\Z$, which is not definable in $(\Z,+,<,0)$. Until very recently $(\Z,+,<,0)$ was the only documented proper dp-minimal expansion of $(\Z,+,0)$ and so, toward Problem \ref{prob}$(b)$ and given the previously quoted results, it is natural to consider expansions of $(\Z,+,0)$ obtained as reducts of $(\Z,+,<,0)$. By Theorem \ref{thm:premain} however, there are no such structures other than $(\Z,+,0)$ and $(\Z,+,<,0)$, and so a positive answer to Question \ref{ques:two}$(b)$ would require a genuinely different dp-minimal expansion of $(\Z,+,0)$. In fact, such expansions have recently been discovered by Alouf and D'Elb\'{e}e \cite{AlDel}.

The following corollary summarizes the progress Theorem \ref{thm:premain} makes with respect to Question \ref{ques:two}$(a)$ (answered in the negative), Question \ref{ques:two}$(b)$, and the more general projects in Problem \ref{prob}.

\begin{corollary}\label{cor:premain}$~$
\begin{enumerate}[$(a)$]
\item Fix $n>0$ and suppose $A\seq\Z^n$ is definable in Presburger arithmetic. Then either $A$ is definable in $(\Z,+,0)$ or, together with the group structure, $A$ defines the ordering on $\Z$ and thus defines any Presburger set.
\item Suppose $\cM$ is a stable expansion of $(\Z,+,0)$. If $A\seq\Z^n$ is definable in both $\cM$ and $(\Z,+,<,0)$, then $A$ is definable in $(\Z,+,0)$.
\item Suppose $\cM$ is a dp-minimal (or even just strong) expansion of $(\Z,+,0)$, which is not interdefinable with $(\Z,+,<,0)$. If $A\seq\Z^n$ is definable in both $\cM$ and $(\Z,+,<,0)$, then $A$ is definable in $(\Z,+,0)$.
\end{enumerate}
\end{corollary}

We again emphasize that part $(c)$ also requires the results cited above from \cite{ADHMS1} and \cite{DoGo} on expansions of $(\Z,+,<,0)$.  The dichotomy in part $(a)$ for Presburger definable sets uncovers a similar phenomenon as a result of Marker \cite{MaComp}: if $S\seq\R^{2n}$ is a semi-algebraic set (i.e. definable in $(\R,+,\cdot,\leq,0,1)$) then, interpreted as a subset of $\C^n$, $S$ is either constructible (i.e. definable in $(\C,+,\cdot,0,1)$) or, together with the complex field structure, $S$ defines any semi-algebraic set.

Finally, we emphasize that the analog of Theorem \ref{thm:premain} for nonstandard models of $\Th(\Z,+,0)$ fails. For example, if $(G,+,<,0)$ is a saturated model of $\Th(\Z,+,<,0)$, then we obtain intermediate structures strictly between $(G,+,0)$ and $(G,+,<,0)$ by expanding $(G,+,0)$ with the ordering on some interval $[0,\alpha]$, where $\alpha\in G^{>0}$ is nonstandard. We can similarly obtain counterexamples using other groups: expand $(\Q,+,0)$ by the ordering on the interval $[0,1]$ to obtain an intermediate structure strictly between $(\Q,+,0)$ and $(\Q,+,<,0)$.

\begin{question}
Let $(G,+,<,0)$ be a model of $\Th(\Z,+,<0)$ or $\Th(\Q,+,<,0)$. Is it true that any proper expansion of $(G,+,0)$, which is a reduct of $(G,+,<,0)$, defines the ordering on some interval? What about other ordered abelian groups?
\end{question}

The remainder of this paper is entirely devoted to the proof of Theorem \ref{thm:premain}. Therefore, we end this introduction with a brief outline of the argument. Section \ref{sec:def} gives basic definitions and reformulates Theorem \ref{thm:premain} as a statement about individual definable sets, namely: if $A\seq\Z^n$ is definable in $(\Z,+,<,0)$ then either $A$ is definable in $(\Z,+,0)$ or the ordering on $\Z$ is definable in the expansion $(\Z,+,0,A)$ (see Theorem \ref{thm:main}). The proof then proceeds by induction on $n$. In Section \ref{sec:1dim} we prove the base case $n=1$, which is a straightforward consequence of quantifier elimination in an appropriate definitional expansion of $(\Z,+,<,0)$. We then turn to several preliminaries needed for the induction step. In Section \ref{sec:module}, we use the classification of sets definable in $(\Z,+,0)$ to give a kind of cell decomposition for $(\Z,+,0)$, where ``cells'' are \emph{quasi-cosets}, i.e. cosets of subgroups with some lower rank piece removed. In Section \ref{sec:poly}, we prove two consequences of a result of Kadets \cite{Kad} on the \emph{inradius} of a polyhedron $P$ in $\R^n$ (i.e. the supremum over the radii of $n$-balls contained in $P$). Polyhedra arise naturally in our situation because of the fact that if $f\colon\Z^n\to \Z$ is definable in $(\Z,+,<,0)$, then there is a decomposition of $\Z^n$ into finitely many subsets definable in $(\Z,+,<,0)$ such that, on each subset, $f$ is the restriction of an affine transformation from $\R^n$ to $\R$ (defined over $\Z$).

Section \ref{sec:technical} begins the technical work toward the induction step of the proof of the main result. In particular, we isolate a subclass of definable sets in $(\Z,+,<,0)$ satisfying certain structural properties, and show that it suffices to only consider sets in this special subclass. Roughly speaking, this subclass is defined by specifying congruence classes, sorting infinite fibers from finite fibers, and uniquely identifying the endpoints of the intervals in fibers with a finite collection of affine transformations. Finally, in Section \ref{sec:proof}, we combine all of these tools to finish the main proof.

\subsection*{Acknowledgements} I would like to thank Dave Marker for first introducing me to this problem, and for many helpful conversations. I also thank Anand Pillay, Sergei Starchenko, Somayeh Vojdani, and Erik Walsberg.

\section{Definitions and reformulation of main result}\label{sec:def}

Given a structure $\cM$ in a first-order language $\cL$, we say that a subset $A\seq M^n$ is \emph{definable in $\cM$} to mean $A$ is $\cL$-definable \emph{with parameters from $M$}.

\begin{definition}
Suppose $\cM_1$ and $\cM_2$ are first-order structures, with the same underlying universe $M$, in first-order languages $\cL_1$ and $\cL_2$, respectively. 
\begin{enumerate}
\item We say $\cM_1$ is a \textbf{reduct} of $\cM_2$ (equivalently, $\cM_2$ is an \textbf{expansion} of $\cM_1$) if, for all $n>0$, every subset of $M^n$, which is definable in $\cM_1$, is also definable in $\cM_2$. 
\item We say $\cM_1$ and $\cM_2$ are \textbf{interdefinable} if $\cM_1$ is a reduct of $\cM_2$ and $\cM_2$ is a reduct of $\cM_1$.
\end{enumerate}
\end{definition}

Let $\Z$ denote the set of integers, and $\N$ the set of nonnegative integers. We let $\Zbar$ denote $\Z\cup\{\nv\infty,\infty\}$, and extend the ordering on $\Z$ to $\Zbar$ in the obvious way.

Given a sequence of sets $(A_1,\ldots,A_n)$, with $A_i\seq\Z^{n_i}$, we define the expansion $(\Z,+,0,A_1,\ldots,A_n)$ of $(\Z,+,0)$ where, with a slight abuse of notation, $A_i$ is an $n_i$-ary relation symbol interpreted as $A_i$.

\begin{definition}
Let $(A_1,\ldots,A_m)$ and $(B_1,\ldots,B_n)$ be sequences of sets, with $A_i\seq\Z^{m_i}$ and $B_i\seq\Z^{n_i}$. 
\begin{enumerate}
\item We say $(A_1,\ldots,A_m)$ is \textbf{definable from} $(B_1,\ldots,B_n)$ if $(\Z,+,0,A_1,\ldots,A_m)$ is a reduct of $(\Z,+,0,B_1,\ldots,B_m)$.
\item We say $(A_1,\ldots,A_m)$ and $(B_1,\ldots,B_n)$ are \textbf{interdefinable} if the structures $(\Z,+,0,A_1,\ldots,A_m)$ and $(\Z,+,0,B_1,\ldots,B_n)$ are interdefinable.
\end{enumerate}
\end{definition}

It is worth noting that, in the previous definition, we have been a little careless with our use of the word \emph{definable}. In particular, it would be more accurate to say, for instance, ``$(A_1,\ldots,A_m)$ is definable from $(B_1,\ldots,B_n)$ \emph{relative to $(\Z,+,0)$}''. However, throughout the paper we focus exclusively on expansions of $(\Z,+,0)$, and therefore always assume that we can use the group language when defining sets. 

Our two main structures are $(\Z,+,0)$ and $(\Z,+,<,0)$. The latter structure is often referred to as \emph{Presburger arithmetic}, since the theory of $(\Z,+,<,0)$ was first formally axiomatized by Presburger \cite{Presb} in 1929. 
We use the following terminology.

\begin{definition}
Fix $n>0$ and a subset $A\seq\Z^n$.
\begin{enumerate}
\item We say $A$ is a \textbf{Presburger set} if $A$ is definable in $(\Z,+,<,0)$.
\item We say $A$ \textbf{defines the ordering} if $\N$ is definable in $(\Z,+,0,A)$.
\end{enumerate}
\end{definition}

We now restate Theorem \ref{thm:premain}.

\begin{theorem}\label{thm:main}
Given $n>0$, if $A\seq\Z^n$ is a Presburger set then either $A$ is definable in $(\Z,+,0)$ or $A$ defines the ordering. In other words, $(\Z,+,0,A)$ is either interdefinable with $(\Z,+,0)$ or with $(\Z,+,<,0)$.
\end{theorem}

It is worth again emphasizing that, for us, ``interdefinable'' means \emph{with parameters}. For example, $1$ is definable in $(\Z,+,<,0)$ without parameters, but not in $(\Z,+,0)$. Of course, any element of $\Z$ is definable from $1$ without parameters, and so one could avoid this issue by replacing $(\Z,+,0)$ with $(\Z,+,0,1)$.

Given $n>0$ and $x,y\in\Z$,  we write $x\equiv_n y$ if $x-y\in n\Z$. For a fixed $n>0$, the binary relation $\equiv_n$ is definable in $(\Z,+,0)$. The main model theoretic ingredients necessary for the proof of Theorem \ref{thm:main}, are the facts that $(\Z,+,-,0,(\equiv_n)_{n>0})$ and $(\Z,+,-,<,0,1,(\equiv_n)_{n>0})$ both have quantifier elimination (see, respectively, Exercise 3.4.6 and Corollary 3.1.21 of \cite{Mabook} for details).  As a result, definable sets in $(\Z,+,0)$ and $(\Z,+,<,0)$ can be described explicitly (see, respectively, Fact \ref{fact:module} and Fact \ref{fact:CD}).

The proof of Theorem \ref{thm:main} will proceed by induction on $n>0$. The base case $n=1$ can be handled easily, and so we will dispense with it right away.

\section{The one-dimensional case}\label{sec:1dim}

In this section, we prove Theorem \ref{thm:main} for the case $n=1$. The main tool is the following description of Presburger subsets of $\Z$. 

\begin{proposition}\label{prop:1dim}
Suppose $A\seq\Z$ is a Presburger set. Then 
$$
A=F\cup \bigcup_{i=1}^k\{a_i+b_in:n\in\N\},
$$
for some finite set $F$, $k\in\N$ (possibly $0$), and $a_1,\ldots,a_k,b_1,\ldots,b_k\in\Z$.
\end{proposition}

The proof is omitted. It follows fairly easily from quantifier elimination for Presburger arithmetic, and even more easily from cell decomposition for Presburger sets (due to Cluckers \cite{Cluck}, see Section \ref{sec:CD}). We have formulated Proposition \ref{prop:1dim} to mimic a recent result of Dolich and Goodrick \cite[Theorem 2.18]{DoGo}, which is a generalization to archimedean ordered abelian groups with \emph{strong} theories.

\begin{definition}
Given $a,b\in\Zbar$ and $m,c\in\N$, with $c<m$, define
$$
[a,b]^c_m=\{x\in\Z:a\leq x\leq b,~x\equiv_m c\}.
$$
\end{definition}

\begin{corollary}\label{cor:1dim}
Suppose $A\seq\Z$ is a Presburger set. Then either $A$ is definable in $(\Z,+,0)$ or $A$ defines the ordering.
\end{corollary}
\begin{proof}
Since any finite set is definable in $(\Z,+,0)$, we may assume $A$ is infinite. Since $A$ is interdefinable with $A\smd F$ for any finite $F\seq\Z$, we may use Proposition \ref{prop:1dim} to assume
$$
A=\bigcup_{i=1}^k(\nv\infty,0]^{c_i}_{m_i}\cup\bigcup_{i=1}^l[0,\infty)^{d_i}_{n_i},
$$
where $k,l\in\N$ (possibly $0$), and $c_i,d_i,m_i,n_i\in\N$ with $c_i<m_i$ and $d_i<n_i$. Let $m$ be the least common multiple of $m_1,\ldots,m_k,n_1,\ldots,n_l$. Then, by adding more terms to the union, we may write
$$
A=\bigcup_{i=1}^{k_*}(\nv\infty,0]^{c_i}_m\cup\bigcup_{i=1}^{l_*}[0,\infty)^{d_i}_m
$$
where $0\leq c_i<m$ and $0\leq d_j<m$ for all $1\leq i\leq k_*$ and $1\leq j\leq l_*$. 

Suppose that there is some $1\leq j\leq l_*$ such that $d_j\neq c_i$ for all $1\leq i\leq k_*$. Then 
$$
x\in\N\miff d_j+mx\in A,
$$
and so $A$ defines the ordering. By a similar argument, if there is some $1\leq i\leq k_*$ such that $c_i\neq d_j$ for all $1\leq j\leq l_*$, then $A$ defines the ordering. So we may assume $\{c_1,\ldots,c_{k_*}\}=\{d_1,\ldots,d_{l_*}\}$. Then
$$
A=\bigcup_{i=1}^{k_*}(\nv\infty,\infty)^{c_i}_m,
$$
which is definable in $(\Z,+,0)$.
\end{proof}

From the proof, we have the following useful observation.

\begin{corollary}\label{cor:bdd}
Suppose $A\seq \Z$ is an infinite Presburger set, which is bounded above or bounded below. Then $A$ defines the ordering.
\end{corollary}

\section{Quasi-coset decomposition in $(\Z,+,0)$}\label{sec:module}

In the induction step of the proof of Theorem \ref{thm:main}, we will use the induction hypothesis to conclude that if a Presburger set $A\seq\Z^{n+1}$ does not define the ordering, then the projection of $A$ to $\Z^n$ is definable in $(\Z,+,0)$. Therefore, in this section, we develop a certain decomposition of subsets of $\Z^n$ definable in $(\Z,+,0)$. We remind the reader that when we say ``definable'', with no other specification, we mean ``definable in $(\Z,+,0)$''.

We start with the following classical fact (see, e.g., \cite{PiGST} or \cite{Prestbook}).

\begin{fact}\label{fact:module}
A set $A\seq\Z^n$ is definable in $(\Z,+,0)$ if and only if it is a (finite) Boolean combination of cosets of subgroups of $\Z^n$.
\end{fact}

The goal of this section is to decompose definable sets in $\Z^n$ into finite unions of definable pieces, such that each piece is ``essentially" a coset (specifically, a coset with some smaller dimensional set removed).

Recall that, given $n>0$, $\Z^n$ is an example of a \emph{free abelian group}, meaning it is an abelian group generated by a linearly independent basis of unique cardinality, called the \emph{rank} of the group. Recall also that any subgroup of a free abelian group is free abelian, which can be shown using the following classical fact (see, e.g., \cite[Theorem 12.4.11]{Artinbook}).

\begin{fact}\label{fact:subbasis}
Suppose $G$ is a free abelian group of rank $n$ and $H\leq G$ is a nontrivial subgroup. Then there is a basis $\{\abar^1,\ldots,\abar^n\}$ of $G$, an integer $k\leq n$, and positive integers $d_1,\ldots,d_k$, such that $d_i$ divides $d_{i+1}$ for all $1\leq i<k$ and $\{d_1\abar^1,\ldots,d_k\abar^k\}$ is a basis for $H$ (in particular, $H$ is free abelian of rank $k$).
\end{fact}

Note that this fact can be used to prove one direction of Fact \ref{fact:module}, namely, that Boolean combinations of cosets are definable. The following corollary of Fact \ref{fact:subbasis} will be useful.

\begin{corollary}\label{cor:finindex}
Suppose $G$ is a free abelian group of rank $n$ and $H\leq G$ is a subgroup of rank $k\leq n$. Then $[G:H]$ is finite if and only if $k=n$.
\end{corollary}

We now define a notion of rank for arbitrary subsets $X\seq\Z^n$. As usual, the rank of the trivial subgroup $\{0\}$ is $0$.

\begin{definition}
$~$
\begin{enumerate}
\item If $C\seq\Z^n$ is a coset of a rank $k$ subgroup of $\Z^n$, then we set $\rk(C)=k$. 
\item Given a nonempty set $X\seq\Z^n$, define the \textbf{rank of $X$}, denoted $\rk(X)$, to be the minimum integer $k\in\{0,1,\ldots,n\}$ such that $X$ is contained in a finite union of cosets of rank at most $k$. Set $\rk(\emptyset)=\nv 1$.
\end{enumerate}
\end{definition}

Using Corollary \ref{cor:finindex}, the reader may check that if $C\seq\Z^n$ is a coset then the rank of $C$, as defined in part (1) of the previous definition, agrees with the rank as defined in part (2). So our notion of rank is well-defined. From the definitions, it follows that $\rk(X\cup Y)=\max\{\rk(X),\rk(Y)\}$ for any $X,Y\seq\Z^n$.

\begin{definition}
Fix  linearly independent $\alpha=\{\abar^1,\ldots,\abar^k\}\subset\Z^n$, and let $G=\langle \abar^1,\ldots,\abar^k\rangle\leq\Z^n$.
\begin{enumerate}
\item Define the group isomorphism $\Phi_\alpha\colon G\to \Z^k$ such that 
$$
\Phi_\alpha(\abar^t)=(0,\ldots,0,1,0,\ldots,0)\in\Z^k,
$$
where $1$ is in the $t\uth$ coordinate.
\item Fix $\cbar\in\Z^k$ and set $C=\cbar+G$. Define the bijection $\Phi_{\alpha,\cbar}\colon C\to \Z^k$ such that $\Phi_{\alpha,\cbar}(\xbar)=\Phi_\alpha(\xbar-\cbar)$.
\end{enumerate}
\end{definition}

Note that if $\alpha$ is a basis for a rank $k$ subgroup $G\leq\Z^n$ then $\Phi_\alpha$ is a an isomorphism of free abelian groups, and therefore can be represented as multiplication by an integer matrix (see \cite[Chapter 12]{Artinbook}). In particular, if $\cbar\in\Z^n$ then $\Phi_{\alpha,\cbar}$ is definable in $(\Z,+,0)$ (as a function on the definable set $C=\cbar+G$).

\begin{remark}
While it will not be necessary for our results, it is also worth pointing out that if $X\seq\Z^n$ is definable in $(\Z,+,0)$, then $\rk(X)=U(X)$ (see \cite{PiGST}). In particular, Lemma \ref{lem:rectcoset} and Corollary \ref{cor:rectcoset}, which discuss rank-preserving properties of definable bijections on definable sets, are very specific manifestations of the fact that $U$-rank is preserved this way in general. To avoid bringing in this extra technology, we give the (short) proofs of these two results.
\end{remark}

\begin{lemma}\label{lem:rectcoset}
Let $G\leq\Z^n$ be a subgroup with basis $\alpha=\{\abar^1,\ldots,\abar^k\}$.  Fix $\cbar\in\Z^n$, and let $C=\cbar+G$. Then the map $D\mapsto \Phi_{\alpha,\cbar}(D)$ is a rank-preserving bijection between the collection of cosets $D$ in $\Z^n$ such that $D\seq C$, and the collection of cosets in $\Z^k$.
\end{lemma}
\begin{proof}
First, if $D$ is a coset in $\Z^n$ and $D\seq C$ then $D=\cbar+\gbar+H$ for some $\gbar\in G$ and subgroup $H\leq G$. A direct calculation then shows $\Phi_{\alpha,\cbar}(D)=\Phi_\alpha(\gbar+H)=\Phi_\alpha(\gbar)+\Phi_\alpha(H)$.  So $\Phi_{\alpha,\cbar}$ is injective from the collection of cosets $D$ in $\Z^n$ such that $D\seq C$, to the collection of cosets in $\Z^k$. Note also that, since $\Phi_\alpha$ is an isomorphism, 
$$
\rk(\Phi_{\alpha,\cbar}(D))=\rk(\Phi_\alpha(H))=\rk(H)=\rk(D).
$$
Finally, to show surjectivity, fix $K\leq\Z^k$ and $\ubar\in\Z^k$. Let $H=\Phi\inv_\alpha(K)\leq G$ and $\gbar=\Phi\inv_\alpha(\ubar)\in G$. Then, setting $D=\cbar+\gbar+H$, we have $D\seq C$ and $\Phi_{\alpha,\cbar}(D)=\ubar+K$. 
\end{proof}

\begin{corollary}\label{cor:rectcoset}
Let $G\leq\Z^n$ be a subgroup with basis $\alpha=\{\abar^1,\ldots,\abar^k\}$.  Fix $\cbar\in\Z^n$, and let $C=\cbar+G$. If $X\seq C$ then $\rk(\Phi_{\alpha,\cbar}(X))=\rk(X)$.
\end{corollary}
\begin{proof}
If $X\seq C$ then $\rk(X)$ and $\rk(\Phi_{\alpha,\cbar}(X))$ are both bounded above by $k=\rk(C)$. Moreover, in the computation of $\rk(X)$ it suffices to only consider cosets of $\Z^n$ which are subsets of $C$. Given $m\leq k$ and a subset $X\seq C$, it follows from Lemma \ref{lem:rectcoset} that $\Phi_{\alpha,\cbar}(X)$ is contained in a finite union of cosets in $\Z^k$ of rank at most $m$ if and only if $X$ is contained in a finite union of cosets in $\Z^n$, which are subsets of $C$ and have rank at most $m$. By definition, $\rk(\Phi_{\alpha,\cbar}(X))=\rk(X)$.
\end{proof}

\begin{definition}
Given $0\leq k\leq n$, a subset $X\seq\Z^n$ is a \textbf{quasi-coset of rank $k$} if $X=C\backslash Z$, where $C$ is a coset of rank $k$ and $Z\seq C$ is definable with $\rk(Z)<k$.
\end{definition}

Note that any quasi-coset in $\Z^n$ is definable in $(\Z,+,0)$ by Fact \ref{fact:module}.

\begin{theorem}\label{thm:qc}
For any $n>0$, if $A\seq\Z^n$ is nonempty and definable in $(\Z,+,0)$, then $A$ can be written as a finite union of quasi-cosets in $\Z^n$.
\end{theorem}
\begin{proof}
For brevity, we call a definable subset $A\seq\Z^k$ \emph{good} if either $A=\emptyset$ or $A$ can be written as a finite union of quasi-cosets in $\Z^n$. By induction on $n\geq 0$, we prove that any definable subset in $\Z^n$ is good (where our convention is $\Z^0=\{0\}$). The base case $n=0$ is trivial. Fix $n>0$ and assume that, for any $k<n$, any definable subset $A\seq\Z^k$ is good.

\noit{Claim 1}: If $X\seq\Z^n$ is definable, with $\rk(X)<n$, then $X$ is good.

\noit{Proof}: Let $k=\rk(X)<n$. Then $X\seq C_1\cup\ldots\cup C_m$, where each $C_i$ is a coset in $\Z^n$ of rank at most $k$. If we can show that each $X\cap C_i$ is good then it will follow that $X$ is good. Therefore, without loss of generality, we assume $X$ is contained in a single coset $C$ of rank $k$. Let $C=\cbar+G$ and fix a basis $\alpha=\{\abar^1,\ldots,\abar^k\}$ for $G$. Let $\Phi=\Phi_{\alpha,\cbar}$. Then $\Phi(X)$ is a definable subset of $\Z^k$, and is therefore good by induction. So we may write
$$
\Phi(X)=E_1\backslash W_1\cup\ldots\cup E_m\backslash W_m,
$$
where $E_i$ is a coset in $\Z^k$ of rank $k_i\leq k$, and $W_i\seq E_i$ is definable with $\rk(W_i)=l_i<k_i$. Let $D_i=\Phi\inv(E_i)$ and $Z_i=\Phi\inv(W_i)$. Then
$$
X=D_1\backslash Z_1\cup\ldots\cup D_n\backslash Z_m.
$$
By Lemma \ref{lem:rectcoset}, each $D_i$ is a coset in $\Z^n$ of rank $k_i$. By Corollary \ref{cor:rectcoset}, each $Z_i$ is a definable subset of $D_i$ of rank $l_i<k_i$. Altogether, $X$ is good.\claim

\noit{Claim 2}: Any finite intersection of quasi-cosets in $\Z^n$ is good.

\noit{Proof}: It suffices to prove the claim for an intersection of two quasi-cosets $X$ and $Y$. Let $A=X\cap Y$. By Claim 1, we may assume $\rk(A)=n$. Let $X=C\backslash U$ and $Y=D\backslash V$ where $C,D$ are cosets in $\Z^n$ and $U,V$ are definable with $U\seq C$, $V\seq D$, $\rk(U)<n$, and $\rk(V)<n$. Since $A$ is nonempty, $C\cap D$ is some coset $E$ in $\Z^n$. Setting $Z=E\cap(U\cup V)$, we have $A=E\backslash Z$. Then $\rk(Z)\leq\max\{\rk(U),\rk(V)\}<n$, which means we must have $\rk(E)=n$. Altogether, $A$ is a quasi-coset of rank $n$, and is therefore good.\claim

\noit{Claim 3}: Let $X\seq\Z^n$ be a quasi-coset. Then $\neg X:=\Z^n\backslash X$ is good.

\noit{Proof}: Set $X=C\backslash Z$, where $C$ is a coset of rank $k\leq n$ and $Z\seq C$ is definable with $\rk(Z)<k$. Then $\neg X=Z\cup \Z^n\backslash C$. Note that $Z$ is good by Claim 1, and so it suffices to show that $\Z^n\backslash C$ is good. If $k<n$ then this is immediate, so we may assume $k=n$. Let $C=\cbar+G$ for some $\cbar\in\Z^n$ and rank $n$ subgroup $G\leq\Z^n$. By Corollary \ref{cor:finindex}, $G$ has finite index in $\Z^n$. So $\Z^n\backslash C$ is a finite union of cosets, and therefore good.\claim

We now proceed with the induction step. Since any coset in $\Z^n$ is good, it suffices by Fact \ref{fact:module} to show that the good sets in $\Z^n$ form a Boolean algebra. Since the union of two good sets is obviously good, it suffices to fix a good set $A\seq\Z^n$ and prove that $\neg A$ is good. Let 
$$
A=X_1\cup\ldots\cup X_n.
$$
where each $X_i$ is a quasi-coset. Then $\neg A=\neg X_1\cap\ldots\cap\neg X_m$. By Claim 3, we may write $\neg X_i=\bigcup_{j\in J_i}Y^i_j$, where $J_i$ is some finite set and each $Y^i_j$ is a quasi-coset. If $J=J_1\times\ldots\times J_m$ then $\neg A=\bigcup_{\jbar\in J}Y_{\jbar}$ where, given $\jbar\in J$,
$$
Y_{\jbar}=Y^1_{j_1}\cap\ldots\cap Y^m_{j_m}.
$$
By Claim 2, each $Y_{\jbar}$ is good. Altogether, $\neg A$ is good, as desired.
\end{proof}

\section{Polyhedra and inscribed balls}\label{sec:poly}

Throughout the section we work in $\R^n$ for a fixed $n>0$. 

\begin{definition}
$~$
\begin{enumerate}
\item A function $f\colon \R^n\to\R$ is an \textbf{affine transformation} if it is of the form
$$
f(\xbar)=u+\sum_{i=1}^n a_ix_i,
$$
for some $u,a_1,\ldots,a_n\in\R$. If $u=0$ then $f$ is \textbf{linear}. If $a_i=0$ for all $1\leq i\leq n$ then $f$ is \textbf{constant}.

\item Given a non-constant affine transformation $f\colon \R^n\to\R$, we define the associated \textbf{affine hyperplane}
$$
H(f)=\{\xbar\in\R^n:f(\xbar)=0\},
$$
as well as the associated \textbf{half-spaces}
\begin{align*}
H^+(f) &= \{\xbar\in\R^n:f(\xbar)>0\},\\
H^-(f) &= \{\xbar\in\R^n:f(\xbar)<0\},\\
\skov{H}^+(f) &= \{\xbar\in \R^n:f(\xbar)\geq 0\},\\
\skov{H}^-(f) &= \{\xbar\in\R^n:f(\xbar)\leq 0\}.
\end{align*}
For  $\bullet\in\{+,-\}$, we may also use $H^\bullet_1(f):=\skov{H}^\bullet(f)$ and $H^\bullet_0(f):=H^\bullet(f)$.

\item A \textbf{polyhedron} in $\R^n$ is an intersection of finitely many half-spaces.\footnote{We are making a slight departure from standard conventions in allowing polyhedra to be defined by possibly \emph{open} half-spaces. }

\item Given a non-constant linear function $f(\xbar)=\sum_{i=1}^n a_ix_i$, and $u<v$ in $\R$, we define the \textbf{plank}
$$
S(f,u,v)=\{\xbar\in\R^n:u\leq f(\xbar)\leq v\}=\skov{H}^+(f-u)\cap\skov{H}^-(f-v).
$$
The \textbf{thickness} of the plank $S(f,u,v)$ is $\frac{u-v}{|\abar|}$, i.e., the distance between the hyperplanes $H(f-u)$ and $H(f-v)$.

\item Given $\xbar\in\R^n$ and $r\geq 0$, we define the \textbf{closed ball}
$$
B_r(\xbar)=\{\ybar\in\R^n:d(\xbar,\ybar)\leq r\}.
$$

\item The \textbf{inradius} of a convex set $P\seq\R^n$ is
$$
r(P)=\sup\{r\geq 0:B_r(\xbar)\seq P\text{ for some }\xbar\in\R^n\}\in\R^{\geq0}\cup\{\infty\}.
$$
\end{enumerate}
\end{definition}

For example, if $B$ is a ball of radius $r$ then $r(B)=r$. If $S$ is a plank of thickness $t$ then $r(S)=\frac{t}{2}$. Moreover, if $P\seq\R^n$ is convex then $r(P)=0$ if and only if $P$ has no interior. In particular, any hyperplane, or more generally, any affine translation of a proper vector subspace of $\R^n$, is a closed polyhedron with inradius $0$.

In this section, we derive some consequences, to be used later, of the following result on convex sets and inscribed balls.

\begin{fact}\label{fact:kad}\textnormal{(Kadets \cite{Kad} 2005)} 
Suppose $Q,P_1,\ldots,P_k$ are closed convex sets in $\R^n$, with nonempty interior. If $Q\seq P_1\cup\ldots\cup P_k$ then $r(Q)\leq\sum_{i=1}^k r(P_i)$.
\end{fact}

This fact has its roots in a 1932 result of Tarski \cite{Tar32}, which essentially deals with the case in $\R^2$ where the $P_i$ are parallel planks. The generalization of Tarski's result to $\R^n$ was given by Bang \cite{Bang} in 1951. 

The following is an immediate consequence of Fact \ref{fact:kad}.

\begin{corollary}\label{cor:kad}
Suppose $Q,P_1,\ldots,P_k$ are polyhedra in $\R^n$ and $Q\seq P_1\cup\ldots\cup P_k$. If $r(Q)=\infty$ then $r(P_i)=\infty$ for some $1\leq i\leq k$.
\end{corollary}

\begin{remark}
Although we quote Corollary \ref{cor:kad} as a consequence of Kadets' result, it is worth observing that we are not using the full power of Fact \ref{fact:kad}, but only the difference between finite and infinite inradius. We leave it as an exercise to the reader to deduce Corollary \ref{cor:kad} directly from basic principles (such an argument was communicated to us by Solecki).
\end{remark}

For later results, we need two applications of Corollary \ref{cor:kad}. First, we show that if $P$ is a polyhedron of infinite inradius, defined by some finite intersection of half-spaces, then the ``opposite'' polyhedron, defined by the intersection of the opposite sides of these half-spaces, also has infinite inradius. 

\begin{lemma}\label{lem:polyinv}
Suppose $f_1,\ldots,f_k$ are linear functions from $\R^n$ to $\R$ and $b_1,\ldots,b_k\in\R$. Fix $\eta\colon \{1,\ldots,k\}\to \{0,1\}$ and define
$$
P^+=\bigcap_{i=1}^kH^+_{\eta(i)}(f_i-b_i)\mand P^-=\bigcap_{i=1}^k H^-_{\eta(i)}(f_i-b_i).
$$
Then $r(P^+)=\infty$ if and only if $r(P^-)=\infty$.
\end{lemma}
\begin{proof}
Without loss of generality, it suffices to assume $r(P^+)=\infty$ and prove $r(P^-)=\infty$. Define
$$
Q=\bigcap_{i=1}^kH^-_{\eta(i)}(f_i+b_i).
$$
For any $\xbar\in\R^n$, $\xbar\in P^+$ if and only if $\nv\xbar\in Q$. Since $\xbar\mapsto\nv\xbar$ is an isometry of $\R^n$, it follows that $r(Q)=r(P^+)=\infty$. 

Given $1\leq i\leq k$, define
$$
S_i=\begin{cases}
S(f_i,\nv b_i,b_i)\backslash H(f-b_i) & \text{if $b_i\geq 0$ and $\eta(i)=0$}\\
S(f_i,\nv b_i,b_i) & \text{if $b_i\geq 0$ and $\eta(i)=1$}\\
S(f_i,b_i,\nv b_i) & \text{if $b_i<0$ and $\eta(i)=0$}\\
S(f_i,b_i,\nv b_i)\backslash H(f_i-b_i) & \text{if $b_i<0$ and $\eta(i)=1$}.
\end{cases}
$$
Then, for any $1\leq i\leq k$, we have
$$
H_{\eta(i)}^-(f_i-b_i)=\begin{cases}
H_{\eta(i)}^-(f_i+b_i)\cup S_i & \text{if $b_i\geq 0$}\\
H_{\eta(i)}^-(f_i+b_i)\backslash S_i & \text{if $b_i<0$}.
\end{cases}
$$
Therefore, we may write
$$
P^-=A\cup\bigcap_{b_i\geq 0}H_{\eta(i)}^-(f_i+b_i)\cup\bigcap_{b_i<0}H_{\eta(i)}^-(f_i+b_i)\backslash S,
$$
where $A$ is some subset of $\R^n$ and $S=\bigcup_{b_i<0}S_i$. In particular, $Q\seq P^-\cup S$. Since $S$ is a finite union of polyhedra, each of finite inradius, and $r(Q)=\infty$, it follows from Corollary \ref{cor:kad} that $r(P^-)=\infty$.
\end{proof}

Finally, we show that the integer points on a polyhedron of infinite inradius cannot be covered by finitely many polyhedra of finite inradii, even after removing some definable small-rank subset of the integer points.

\begin{lemma}\label{lem:polyint}
Suppose $X\seq\Z^n$ is definable in $(\Z,+,0)$, with $\rk(X)<n$. Let $P\seq\R^n$ be a a polyhedron, with $r(P)=\infty$. Suppose $Q_1,\ldots,Q_k$ are polyhedra in $\R^n$ such that $P\cap(\Z^n\backslash X)\seq Q_1\cup\ldots\cup Q_k$. Then $r(Q_i)=\infty$ for some $1\leq i\leq k$.
\end{lemma}
\begin{proof}
By definition of rank, there are cosets $C_1,\ldots,C_l\seq\Z^n$ such that $X\seq C_1\cup\ldots\cup C_l$ and, for all $1\leq i\leq l$, $\rk(C_i)=n_i<n$. For $1\leq i\leq l$, let $C_i=\cbar^i+G_i$ for some element $\cbar^i\in\Z^n$ and subgroup $G_i\leq \Z^n$ of rank $n_i$. Let $V_{G_i}$ be the convex closure of $G_i$ in $\R^n$, and note that $V_{G_i}$ is an $n_i$-dimensional subspace of $\R^n$. Set $R_i=\cbar+V_{G_i}$. Then $R_i$ is a polyhedron with $r(R_i)=0$ and $C_i\seq R_i$. Altogether,
\begin{equation*}
P\cap\Z^n\seq Q_1\cup\ldots\cup Q_k\cup R_1\cup\ldots\cup R_l.\tag{$\dagger$}
\end{equation*}

Using an inductive argument, it is straightforward to show that the collection of subsets of $\R^n$, which can be written as a finite union of polyhedra, is a Boolean algebra. Therefore, we may write
\begin{equation*}
P\backslash(Q_1\cup\ldots\cup Q_k\cup R_1\cup\ldots\cup R_l)=S_1\cup\ldots\cup S_m,\tag{$\dagger\dagger$}
\end{equation*}
for some polyhedra $S_1,\ldots,S_m$. Note that $S_i\seq P$ for all $1\leq i\leq m$, and we have
$$
P\seq Q_1\cup\ldots\cup Q_k\cup R_1\cup\ldots\cup R_l\cup S_1\cup\ldots\cup S_m.
$$
By Corollary \ref{cor:kad}, and since $r(R_i)=0$ for all $1\leq i\leq l$, it follows that either $r(Q_i)=\infty$ for some $1\leq i\leq k$ or $r(S_i)=\infty$ for some $1\leq i\leq m$. Suppose, toward a contradiction, that $r(S_i)=\infty$ for some $1\leq i\leq m$. Then we may find a point $\xbar\in S_i\cap\Z^n$. But then $\xbar\in S_i\cap(P\cap\Z^n)$, which contradicts $(\dagger)$ and $(\dagger\dagger)$.
\end{proof}

\section{Cell decomposition for Presburger sets}\label{sec:CD}

In this section, we briefly summarize the cell decomposition of sets definable in $(\Z,+,<,0)$, which is a result of Cluckers \cite{Cluck}. Recall that $\Zbar=\Z\cup\{\nv\infty,\infty\}$.

\begin{definition}
Fix $n>0$.
\begin{enumerate}
\item A function $f\colon \Z^n\to\Zbar$ is \textbf{extreme} if $f(\xbar)=\infty$ for all $\xbar\in\Z^n$ or $f(\xbar)=\nv\infty$ for all $\xbar\in\Z^n$. If $f$ is extreme and $c\in\Z$ then, by convention, we let $f+c=f$.
\item A partial function $f\colon \Z^n\to\Z$ is \textbf{standard $\Z$-linear} if there are $\mbar,\cbar\in\N^n$ such that $c_i<m_i$ for all $1\leq i\leq n$, $\dom(f)=\cbar+\mbar\Z^n$, and
$$
f(\xbar)=u+\sum_{i=1}^na_i\left(\frac{x_i-c_i}{m_i}\right),
$$
for some $u\in\Z$ and $\abar\in\Z^n$. 
\item A partial function $f\colon \Z^n\to\Zbar$ is \textbf{$\Z$-linear} if it is either extreme or standard $\Z$-linear. 
\end{enumerate}
\end{definition}

\begin{definition}\label{def:cells}
Given $\eta\in 2^{<\omega}\backslash\{\emptyset\}$, we define, by induction on $|\eta|$, the notion of an \textbf{$\eta$-cell} in $\Z^{|\eta|}$.
\begin{enumerate}
\item A \textbf{$(0)$-cell} is a singleton $\{x\}$ for some $x\in\Z$. 
\item A \textbf{$(1)$-cell} is an infinite set of the form 
$$
[a,b]^c_m:=\{x\in\Z:a\leq x\leq b,~x\equiv_m c\}.
$$
where $a,b\in\Zbar$ and $m,c\in\N$ with $c<m$.
\item Given $\eta\in 2^{<\omega}\backslash\{\emptyset\}$, an \textbf{$(\eta,0)$-cell} is a set of the form
$$
X[f]:=\{(\xbar,y)\in\Z^{|\eta|+1}:\xbar\in X,f(\xbar)=y\},
$$
where $X\seq\Z^{|\eta|}$ is an $\eta$-cell and $f$ is a standard $\Z$-linear function with $X\seq\dom(f)$. 
\item Given $\eta\in 2^{<\omega}\backslash\{\emptyset\}$, an \textbf{$(\eta,1)$-cell} is a set of the form
$$
X[f,g]^c_m:=\{(\xbar,y)\in\Z^{|\eta|+1}:\xbar\in X,~y\in[f(\xbar),g(\xbar)]^c_m\},
$$
where $X\seq\Z^{|\eta|}$ is an $\eta$-cell, $f,g$ are $\Z$-linear functions such that $X\seq\dom(f)\cap\dom(g)$, $f(\xbar)\leq g(\xbar)$ for all $\xbar\in X$, $c,m\in\N$ with $c<m$, and there is no uniform finite bound on the sets $[f(\xbar),g(\xbar)]^c_m$ for $\xbar\in X$.
\end{enumerate}
\end{definition}

The following is (a slightly weaker version of) Cluckers' cell decomposition \cite{Cluck}.

\begin{fact}\label{fact:CD}
Any Presburger set in $\Z^n$ can be written as a finite union of cells in $\Z^n$.
\end{fact}

The distinction between $(\eta,0)$-cells and $(\eta,1)$-cells will not be essential for our work. In particular, if $X[f]$ is an $(\eta,0)$-cell then, since $X[f]=X[f,f]^0_1$, we can view $X[f]$ as a ``flat'' $(\eta,1)$-cell. Precisely, all we will need from cell decomposition is the following corollary.

\begin{corollary}\label{cor:CD}
Given $n>0$, if $A\seq\Z^{n+1}$ is a Presburger set then
$$
A=\bigcup_{t=1}^kX_t[f_t,g_t]^{c_t}_{m_t},
$$
where, for all $1\leq t\leq k$, $X_t\seq\Z^n$ is a Presburger set, $f_t,g_t$ are $\Z$-linear functions such that $X_t\seq\dom(f_t)\cap\dom(g_t)$, and $m_t,c_t\in\N$ with $c_t<m_t$.
\end{corollary}

\begin{remark}
A crucial observation is that any standard $\Z$-linear function is \emph{definable in $(\Z,+,0)$}. We will use this fact constantly in the subsequent arguments, and without further mention.
\end{remark}

\section{Technical reductions toward the main proof}\label{sec:technical}

In this section, we develop the technical tools necessary to prove the induction step of Theorem \ref{thm:main}. Therefore, throughout the section, we fix an integer $n>0$ and consider Presburger sets in $\Z^{n+1}$.

The goal of this section is to isolate a subclass of Presburger sets in $\Z^{n+1}$, satisfying certain structural assumptions, such that, in order to prove the induction step of Theorem \ref{thm:main}, it suffices to only consider sets in this special subclass (see Corollary \ref{cor:red3} for the precise statement). Roughly speaking, we apply a series of three reductions, starting with a Presburger set $A\seq \Z^{n+1}$ of the form given by Corollary \ref{cor:CD}. Our first reduction will be to ``sort'' the congruence classes $c_t~(\mod m_t)$ so that we may assume all $m_t$'s are the same and all $c_t$'s are $0$. The second reduction is to separate infinite fibers from finite fibers, and show that it suffices to assume $A$ has finite fibers. The third reduction is to identify the endpoints of the intervals in the fibers of $A$, and show that it suffices to assume there is a single finite collection of standard $\Z$-linear functions, which, up to some permutation, precisely determine the endpoints of any fiber of $A$.

For the sake of brevity, we will use the following terminology.

\begin{definition}
For any $n>0$, a Presburger set $A\seq\Z^n$ is \textbf{peripheral} if either $A$ is definable in $(\Z,+,0)$ or $A$ defines the ordering.
\end{definition}

Also, given an integer $k>0$, we let $[k]$ denote the set $\{1,\ldots,k\}$.

\subsection{Sorting congruence classes}

To simplify notation, given $X[f,g]^c_m$ as in Definition \ref{def:cells}, if $c=0$ then we omit it and write $X[f,g]_m$.

\begin{definition}
A set $A\seq\Z^{n+1}$ is \textbf{uniformly congruent} if
$$
A=\bigcup_{t=1}^kX_t[f_t,g_t]_m,
$$
where $m>0$ and, for all $t\in[k]$, $X_t\seq\Z^n$ is a Presburger set and $f_t,g_t$ are $\Z$-linear functions such that $X_t\seq\dom(f_t)\cap\dom(g_t)$.
\end{definition}

\begin{proposition}
Suppose $A\seq\Z^{n+1}$ is a Presburger set. Then $A$ is interdefinable with a finite sequence of uniformly congruent Presburger sets in $\Z^{n+1}$.
\end{proposition}
\begin{proof}
By Corollary \ref{cor:CD}, we may write
$$
A=\bigcup_{t=1}^kX_t[f_t,g_t]^{c_t}_{m_t},
$$
where, for all $1\leq t\leq k$, $X_t\seq\Z^n$ is a Presburger set, $f_t,g_t$ are $\Z$-linear functions such that $X_t\seq\dom(f_t)\cap\dom(g_t)$, and $m_t,c_t\in\N$ with $c_t<m_t$. Let $m$ be the least common multiple of $m_1,\ldots,m_k$. Given $0\leq d<m$, we let 
$$
A_d=\{(\xbar,y)\in A:y\equiv_m d\}.
$$
Then $A=\bigcup_{0\leq d<m}A_d$ and each $A_d$ is definable from $A$.

\noit{Claim}: Given $0\leq d<m$ there is a set $I_d\seq [k]$ such that
$$
A_d=\bigcup_{t\in I_d}X_t[f_t,g_t]^d_m.
$$

\noit{Proof}: Let $I_d=\{t\in [k]:A_d\cap X_t[f_t,g_t]^{c_t}_{m_t}\neq\emptyset\}$ and, for $t\in I_d$, let $B_t=A_d\cap X_t[f_t,g_t]^{c_t}_{m_t}$. Then $A_d=\bigcup_{t\in I_d}B_t$. Fix $t\in I_d$. We immediately have $B_t\seq X_t[f_t,g_t]^d_m$. By assumption, there is $(\xbar,z)\in B_t$ and so $z\equiv_{m_t}c_t$ and $z\equiv_m d$. Since $m_t$ divides $m$, it follows that $d\equiv_{m_t}c_t$. Therefore, if $(\xbar,y)\in X_t[f_t,g_t]^d_m$ then $y\equiv_m d$, and so $y\equiv_{m_t}c_t$, which means $(\xbar,y)\in X_t[f_t,g_t]^{c_t}_{m_t}\cap A_d=B_t$. Altogether, $B_t=X_t[f_t,g_t]^d_m$. \claim

Now, for $0\leq d<m$, set $C_d=\{(\xbar,y-d):(\xbar,y)\in A_d\}$. Then $C_d$ is interdefinable with $A_d$, and so, altogether, $A$ is interdefinable with $(C_d)_{0\leq d<m}$. Moreover,
$$
C_d=\bigcup_{t\in I_d}X_t[f_t-d,g_t-d]_m.
$$
In particular, each $C_d$ is uniformly congruent.
\end{proof}

From this result, we obtain our first reduction. 

\begin{corollary}\label{cor:red1}
If every uniformly congruent Presburger set in $\Z^{n+1}$ is peripheral, then every Presburger set in $\Z^{n+1}$ is peripheral. 
\end{corollary}

\subsection{Eliminating infinite fibers}

Similar to before, given $a,b\in\Zbar$ and $m>0$, we let $[a,b]_m=[a,b]^0_m$. 

\begin{definition}\label{def:wsf}
Let $A\seq\Z^{n+1}$ be a Presburger set.
\begin{enumerate}
\item Let $\pi(A)=\{\xbar\in\Z^n:(\xbar,y)\in\Z^{n+1}\text{ for some }y\in\Z\}$ be the \textbf{projection of $A$ to $\Z^n$}.
\item Given $\xbar\in\Z^n$, define the \textbf{fiber} $A_{\xbar}=\{y\in\Z:(\xbar,y)\in A\}$. In particular, $A_{\xbar}\neq\emptyset$ if and only if $\xbar\in\pi(A)$.
\item  We say $A$ has \textbf{weakly sorted fibers} if there are sets $F=\{f_1,\ldots,f_k\}$ and $G=\{g_1,\ldots,g_l\}$ of standard $\Z$-linear functions, and an integer $m>0$, such that for all $\xbar\in\pi(A)$, there is some $I(\xbar)\seq [k]\times[l]$ satisfying:
\begin{enumerate}[$(i)$]
\item $\xbar\in \dom(f_s)\cap\dom(g_t)$ and $f_s(\xbar)\leq g_t(\xbar)$ for all $(s,t)\in I(\xbar)$, and
\item $A_{\xbar}=\bigcup_{(s,t)\in I(\xbar)}[f_s(\xbar),g_t(\xbar)]_m$.
\end{enumerate}
We may also specify that $A$ has weakly sorted fibers \textbf{witnessed by $(F,G,m)$}.
\end{enumerate}
\end{definition}

Note that, since we assume the functions in $F$ and $G$ are standard, it follows that if $A\seq\Z^{n+1}$ has weakly sorted fibers then $A_{\xbar}$ is finite for all $\xbar\in\Z^n$.

\begin{proposition}
Assume every Presburger set in $\Z^n$ is peripheral. Suppose $A\seq\Z^{n+1}$ is a uniformly congruent Presburger set. Then $A$ is interdefinable with a finite sequence of Presburger sets in $\Z^{n+1}$ with weakly sorted fibers.
\end{proposition}
\begin{proof}
First, if $A$ defines the ordering then it is interdefinable with $\N\times\{0\}^n$, which has weakly sorted fibers. So we may assume $A$ does not define the ordering.

By assumption $A=\bigcup_{t=1}^kX_t[f_t,g_t]_m$ for some $X_t,f_t,g_t$ and $m>0$. Note that all fibers of $A$ are definable from $A$, and are therefore Presburger definable subsets of $\Z$, which do not define the ordering. It follows from Corollary \ref{cor:bdd} that any infinite fiber of $A$ is unbounded above and below. 

Let $X=\pi(A)=\bigcup_{t=1}^kX_t$, and note that $X$ is definable from $A$. Given $\xbar\in X$, let $I(\xbar)=\{(t,t):\xbar\in X_t\}$. In particular, if $\xbar\in X$ then $\xbar\in\dom(f_t)\cap\dom(g_t)$ for all $(t,t)\in I(\xbar)$ and, moreover,
\begin{equation*}
A_{\xbar}=\bigcup_{(t,t)\in I(\xbar)}[f_t(\xbar),g_t(\xbar)]_m.\tag{$\dagger$}
\end{equation*}
Set $Y_1=\{\xbar\in X:A_{\xbar}\text{ is infinite}\}$ and $Y_2=X\backslash Y_1$.

\noit{Claim}: For any $\xbar\in\Z^n$, $\xbar\in Y_1$ if and only if $A_{\xbar}+A_{\xbar}=m\Z$.

\noit{Proof}: Clearly, if $A_{\xbar}+A_{\xbar}=m\Z$ then $A_{\xbar}$ is infinite and so $\xbar\in Y_1$. Conversely, suppose $\xbar\in Y_1$. As previously noted, it follows from Corollary \ref{cor:bdd} that $A_{\xbar}$ is unbounded above and below.  Combined with $(\dagger)$, there must be $s,t\in I(\xbar)$ such that $f_s=\nv\infty$ and $g_t=\infty$, which means there are $a,b\in\Z$ such that $A_{\xbar}\cup[a,b]_m=m\Z$. Fix $z\in m\Z$ and choose $y\in m\Z$ such that $y>\max\{b,z-a\}$. If $x=z-y$ then $x,y\in A_{\xbar}$ and $x+y=z$. Altogether, $A_{\xbar}+A_{\xbar}=m\Z$.\claim
 
It follows from the proof of the claim that if $\xbar\in Y_1$ then $A_{\xbar}$ is cofinite in $m\Z$. By the claim, $Y_1$ is definable from $A$, and therefore so is $Y_2$. For $i\in\{1,2\}$, let 
$$
A_i=\{(\xbar,y)\in A:\xbar\in Y_i\}.
$$
Then $A=A_1\cup A_2$ and so $A$ is interdefinable with $(A_1,A_2)$.

We first show that $A_2$ has weakly sorted fibers. Let $F$ be the set of standard $f_t$, and $G$ the set of standard $g_t$. By $(\dagger)$, if $\xbar\in Y_2=\pi(A_2)$ and $(t,t)\in I(\xbar)$, then $f_t$ and $g_t$ are standard. Therefore $A_2$ has weakly sorted fibers, witnessed by $(F,G,m)$. 

To finish the proof, we assume $A_1\neq\emptyset$ (i.e. $Y_1\neq\emptyset$) and show that $A_1$ is interdefinable with a set $B\seq\Z^{n+1}$ with weakly sorted fibers. In particular, let $B=(Y_1\times m\Z)\backslash A_1$ (and so $A_1=(Y_1\times m\Z)\backslash B$). Since $Y_1\seq\Z^n$ is definable from $A$, and $A$ does not define the ordering, it follows that $Y_1$ does not define the ordering and therefore is definable in $(\Z,+,0)$ by assumption. Therefore $A_1$ and $B$ are interdefinable. 

Let $F^*=\{f_t-1:f_t\in F\}$ and $G^*=\{g_t+1:g_t\in G\}$. If $\xbar\in \pi(B)\seq X_1$ then, using $(\dagger)$, we have
$$
B_{\xbar}=m\Z\backslash A_{\xbar}=\bigcap_{(t,t)\in I(\xbar)}\bigg([\nv\infty,f_t(\xbar)-1]_m\cup[g_t(\xbar)+1,\infty]_m\bigg).
$$
On the other hand, if $\xbar\in \pi(B)$ then $A_{\xbar}$ is cofinite in $m\Z$, and so $B_{\xbar}$ is finite. Therefore, $(G^*,F^*,m)$ witnesses that $B$ has weakly sorted fibers.
\end{proof}

Combined with Corollary \ref{cor:red1}, we obtain our next reduction.

\begin{corollary}\label{cor:red2}
Assume every Presburger set in $\Z^n$ is peripheral. If every Presburger set in $\Z^{n+1}$, with weakly sorted fibers, is peripheral, then every Presburger set in $\Z^{n+1}$ is peripheral.
\end{corollary}

\subsection{Identifying fibers up to permutation of boundary points}

Given $k>0$, let $S_k$ denote the group of permutations of $[k]$.

\begin{definition}
Fix $m>0$.
\begin{enumerate}
\item Define the binary relations $<_m$ and $\lhd_m$ on $\Z$ by
\begin{align*}
a<_m b &\miff a<x<b\text{ for some }x\in m\Z,\\
a\lhd_m b &\miff a\leq x\leq b \text{ for some }x\in m\Z.
\end{align*}
\item Given $x\in\Z$, let $\rho_m(x)$ be the unique element of $\{0,1,\ldots,m-1\}$ such that $x\equiv_m \rho_m(x)$. Define functions $L_m,R_m,L^-_m,R^-_m$ from $\Z$ to $m\Z$ such that, for $x\in \Z$,
\begin{alignat*}{2}
L_m(x) &:= \max\{y\in m\Z:y\leq x\} & &= x-\rho_m(x),\\
R_m(x) &:= \min\{y\in m\Z:x\geq y\} & &= x+\rho_m(\nv x),\\
L^-_m(x) &:= \max\{y\in m\Z:y<x\} & &= L_m(x)-m,\\
R^+_m(x) &:= \max\{y\in m\Z:y>x\}  & &= R_m(x)+m.
\end{alignat*}
Note that $R_m$, $L^-_m$, and $R^-_m$ are all definable in $(\Z,+,0)$.
\item Suppose $A\seq\Z^{n+1}$ is is a Presburger set. We say $A$ has \textbf{sorted fibers} if there are tuples $\fbar=(f_1,\ldots,f_k)$ and $\gbar=(g_1,\ldots,g_k)$ of standard $\Z$-linear functions on $\Z^n$, and an integer $m>0$, such that  $\pi(A)\seq \dom(f_t)\cap\dom(g_t)$ for all $t\in[k]$ and, for all $\xbar\in\pi(A)$, there are $\sigma,\tau\in S_k$ satisfying:
\begin{enumerate}[$(i)$]
\item $f_{\sigma(t)}(\xbar)\lhd_m g_{\tau(t)}(\xbar)$ for all $1\leq t\leq k$,
\item $g_{\tau(t)}(\xbar)<_m f_{\sigma(t+1)}(\xbar)$ for all $1\leq t<k$,
\item $A_{\xbar}=\bigcup_{t=1}^k[f_{\sigma(t)}(\xbar),g_{\tau(t)}(\xbar)]_m$.
\end{enumerate}
We may also specify that $A$ has sorted fibers \textbf{witnessed by $(\fbar,\gbar,m)$}.
\end{enumerate}
\end{definition}

\begin{proposition}\label{prop:wsf}
Suppose $A\seq\Z^{n+1}$ is a Presburger set with weakly sorted fibers, witnessed by $(F,G,m)$. Then $A$ is interdefinable with a finite sequence $(B_i)$ of Presburger sets in $\Z^{n+1}$ such that each $B_i$ has sorted fibers witnessed by some $(\fbar^i,\gbar^i,m)$, with $|\fbar^i|=|\gbar^i|\leq\min\{|F|,|G|\}$.
\end{proposition}
\begin{proof}
Let $F=\{f_1,\ldots,f_k\}$ and $G=\{g_1,\ldots,g_l\}$. For $\xbar\in\pi(A)$, let $I(\xbar)\seq[k]\times[l]$ be as in Definition \ref{def:wsf}$(3)$, and let $I_1(\xbar)$ and $I_2(\xbar)$ be the projections to the first and the second coordinate, respectively. Define the sets
$$
A^-=\{(\xbar,y)\in A:(\xbar,L^-_m(y))\not\in A\}\mand A^+=\{(\xbar,y)\in A:(\xbar,R^+_m(y))\not\in A\}.
$$
Note that $A^-$ and $A^+$ are each definable from $A$.

\noit{Claim}: For any $\xbar\in\Z^n$, 
$$
A^-_{\xbar}\seq\{R_m(f_s(\xbar)):s\in I_1(\xbar)\}\mand A^+_{\xbar}\seq\{L_m(g_t(\xbar)):t\in I_2(\xbar)\}.
$$

\noit{Proof}: Fix $\xbar\in\Z^n$. Suppose $y\in A^-_{\xbar}$. Then $y\in[f_s(\xbar),g_t(\xbar)]_m$ for some $(s,t)\in I(\xbar)$. Since $(\xbar,L^-_m(y))\not\in A$, we must have $L^-_m(y)<f_s(\xbar)\leq y$. Since $y\in m\Z$, it follows that $y=R_m(f_s(\xbar))$. The proof of the second containment is similar.\claim

Let $\Sigma=\cP([k])\times\cP([l])$, where $\cP$ denotes powerset. Given $(\alpha,\beta)\in\Sigma$, define $Y^\beta_\alpha\seq\Z^k$ such that $\xbar\in Y^\beta_\alpha$ if and only if:
\begin{enumerate}[$(i)$]
\item $\xbar\in\pi(A)$, $\alpha\seq I_1(\xbar)$, and $\beta\seq I_2(\xbar)$,
\item $A^-_{\xbar}=\{R_m(f_t(\xbar)):t\in\alpha\}$,
\item $A^+_{\xbar}=\{L_m(g_t(\xbar)):t\in\beta\}$,
\item for all distinct $s,t\in\alpha$, $R_m(f_s(\xbar))\neq R_m(f_t(\xbar))$, and
\item for all distinct $s,t\in\beta$, $L_m(g_s(\xbar))\neq L_m(g_t(\xbar))$.
\end{enumerate}
Then $Y^\beta_\alpha$ is definable from $A$ and, moreover, it follows from the claim that
$$
\pi(A)=\bigcup_{(\alpha,\beta)\in\Sigma}Y^\beta_\alpha.
$$
Given $(\alpha,\beta)\in\Sigma$, set $B^{\beta}_\alpha=\{(\xbar,y)\in A:\xbar\in Y^\beta_\alpha\}$. Let $\Sigma_0=\{(\alpha,\beta)\in\Sigma:B^\beta_\alpha\neq\emptyset\}$. Then 
$$
A=\bigcup_{(\alpha,\beta)\in\Sigma_0}B^\beta_\alpha,
$$
and so $A$ is interdefinable with $(B^\beta_\alpha)_{(\alpha,\beta)\in\Sigma}$. Given $\alpha\seq[k]$ and $\beta\seq[l]$, define $\fbar^\alpha=(f_s)_{s\in\alpha}$ and $\gbar^\beta=(g_t)_{t\in\beta}$. To finish the proof, we fix $(\alpha,\beta)\in\Sigma_0$ and show $|\alpha|=|\beta|$ and $B^\beta_\alpha$ has sorted fibers, witnessed by $(\fbar^\alpha,\gbar^\beta,m)$. To ease notation, let $Y=Y^\beta_\alpha$ and $B=B^\beta_\alpha$. Without loss of generality, we also assume $\alpha=[p_1]$ and $\beta=[p_2]$ for some $p_1\leq k$ and $p_2\leq l$. 

Fix $\xbar\in Y$. We have
$$
B_{\xbar}=A_{\xbar}=\bigcup_{(s,t)\in I(\xbar)}[f_s(\xbar),g_t(\xbar)]_m.
$$
Therefore we can write $B_{\xbar}$ as a union
$$
B_{\xbar}=\bigcup_{t=1}^p[a_t,b_t]_m,
$$
where $\{a_1,\ldots,a_p\}\seq\{f_1(\xbar),\ldots,f_k(\xbar)\}$, $\{b_1,\ldots,b_p\}\seq\{g_1(\xbar),\ldots,g_l(\xbar)\}$, and
$$
a_1\lhd_m b_1<_m a_2\lhd_m b_2<_m\ldots <_m a_p\lhd_m b_p.
$$
Altogether, we want to show $p_1=p=p_2$ and there are $\sigma,\tau\in S_p$ such that, for all $1\leq t\leq p$, $a_t=f_{\sigma(t)}(\xbar)$ and $b_t=g_{\tau(t)}(\xbar)$. To see this, first observe that $R_m(a_1),\ldots,R_m(a_p)$ are distinct elements in $A^-_{\xbar}$ and $L_m(b_1),\ldots,L_m(b_p)$ are distinct elements in $A^+_{\xbar}$. Moreover, if $t\in\alpha$ then $R_m(f_t(\xbar))\in A^-_{\xbar}$, and so $f_t(\xbar)$ is a left endpoint of some interval in 
$$
\cI=\{[a_1,b_1]_m,\ldots,[a_p,b_p]_m\}.
$$
Similarly, if $t\in\beta$ then $L_m(g_t(\xbar))\in A^+_{\xbar}$, and so $g_t(\xbar)$ is a right endpoint of some interval in $\cI$. Therefore,
\begin{align*}
\{R_m(a_1),\ldots,R_m(a_p)\} &= \{R_m(f_t(\xbar)):1\leq t\leq p_1\}\mand\\
 \{L_m(b_1),\ldots,L_m(b_p)\} &= \{L_m(g_t(\xbar)):1\leq t\leq p_2\},
\end{align*}
as desired.
\end{proof}

Combined with Corollary \ref{cor:red2}, we have our third and final reduction.

\begin{corollary}\label{cor:red3}
Assume every Presburger set in $\Z^n$ is peripheral. If every Presburger set in $\Z^{n+1}$, with sorted fibers, is peripheral, then every Presburger set in $\Z^{n+1}$ is peripheral.
\end{corollary}

\subsection{Exchanging parallel functions}

In the proof of the main result, we will use Corollary \ref{cor:red3} to focus on Presburger sets in $\Z^{n+1}$ with sorted fibers. Given such a set $A$, we will assume $A$ does not define the ordering and prove that $A$ is definable in $(\Z,+,0)$. In this section, we prove one final technical result, which will allow us to argue by induction on the number of $\Z$-linear functions used to determine the fibers of $A$.

\begin{lemma}\label{lem:fibers}
Suppose $A\seq\Z^{n+1}$ has sorted fibers, witnessed by $(\fbar,\gbar,m)$, where $\fbar=(f_1,\ldots,f_k)$ and $\gbar=(g_1,\ldots,g_k)$. Assume $\pi(A)$ is definable in $(\Z,+,0)$, and suppose there are $s,t\in \{1,\ldots,k\}$ such that $f_s-g_t$ is constant on $\pi(A)$. 
\begin{enumerate}[$(a)$]
\item Assume $g_t=f_s+c$ on $\pi(A)$, for some $c\geq 0$, and set 
$$
B=\{(\xbar,y)\in A:f_s(\xbar)\leq y\leq g_t(\xbar)\}.
$$
Then $B$ is definable in $(\Z,+,0)$.
\item If $k>1$ then $A$ is interdefinable with a finite sequence of Presburger sets $(B_i)$ in $\Z^{n+1}$ such that each $B_i$ has sorted fibers, witnessed by some $(\fbar^i,\gbar^i,m)$ with $|\fbar^i|=|\gbar^i|<k$.
\end{enumerate}
\end{lemma}
\begin{proof}
Without loss of generality, we may assume $f_k-g_k$ is constant on $\pi(A)$. 

Part $(a)$. Given $1\leq p\leq k$, define the set
$$
I_p=\{(i_1,j_1,\ldots,i_p,j_p)\in\Z^{2p}:0=i_1\leq j_1\leq i_2\leq j_2\leq\ldots\leq i_p\leq j_p=c\}.
$$
We claim that $(\xbar,y)\in B$ if and only if
\begin{enumerate}
\item $\xbar\in\pi(A)$ and $y\in m\Z$, and
\item there are
\begin{enumerate}[$(i)$]
\item an integer $1\leq p\leq k$,
\item injective functions $\sigma,\tau\colon [p]\to [k]$, with $\sigma(1)=k$ and $\tau(p)=k$, and
\item a tuple $(i_1,j_1,\ldots,i_p,j_p)\in I_p$,
\end{enumerate}
such that
\begin{enumerate}[$\bullet$]
\item for all $1\leq t\leq p$, $f_{\sigma(t)}(\xbar)=f_k(\xbar)+i_t$ and $g_{\tau(t)}(\xbar)=f_k(\xbar)+j_t$,
\item for all $s\in [k]\backslash\Im\sigma$ and $t\in [k]\backslash\Im\tau$, neither $f_s(\xbar)$ nor $g_t(\xbar)$ is in $[f_k(\xbar),g_k(\xbar)]$, i.e.,
$$
\bigwedge_{i=0}^c\big[(f_s(\xbar)\neq f_k(\xbar)+i)\wedge (g_t(\xbar)\neq f_k(\xbar)+i)\big],\mand
$$
\item $y\in\bigcup_{t=1}^p[f_{\sigma(t)}(\xbar),g_{\tau(t)}(\xbar)]$, i.e.,
$$
\bigvee_{t=1}^p\bigvee_{i=i_t}^{j_t}y=f_k(\xbar)+i.
$$
\end{enumerate}
\end{enumerate}
In particular, since $\pi(A)$ is definable in $(\Z,+,0)$ by assumption, this shows that $B$ is definable in $(\Z,+,0)$.

To verify that the above data defines $B$, fix $(\xbar,y)\in B$. Then $f_k(\xbar)\leq g_k(\xbar)$ and so, since $A$ has sorted fibers, we may fix some $p\leq k$ for which there are distinct $u_1,\ldots,u_p\in [k]$ and distinct $v_1,\ldots,v_p\in [k]$ such that:
\begin{enumerate}[$(\ast)_1$]
\item $f_k(\xbar)=f_{u_1}(\xbar)\lhd_m g_{v_1}(\xbar)<_m\ldots<_m f_{u_p}(\xbar)\lhd_m g_{u_p}(\xbar)=g_k(\xbar)$,
\item for all $s\not\in\{u_1,\ldots,u_p\}$ and $t\not\in v_1,\ldots,v_p\}$, neither $f_s(\xbar)$ nor $g_t(\xbar)$ is in $[f_k(\xbar),g_k(\xbar)]$, and
\item $A_{\xbar}\cap[f_k(\xbar),g_k(\xbar)]=\bigcup_{t=1}^p[f_{u_t}(\xbar),g_{v_t}(\xbar)]_m$.
\end{enumerate}
Therefore, letting $\sigma:t\mapsto u_t$, $\tau:t\mapsto v_t$, and setting $i_t=f_{u_t}(\xbar)-f_k(\xbar)$ and $j_t=g_{v_t}(\xbar)-f_k(\xbar)$, we have properties $(1)$ and $(2)$ above.

Conversely, suppose $(\xbar,y)\in\Z^{n+1}$ satisfies properties $(1)$ and $(2)$ above, witnessed by $1\leq p\leq k$, $\sigma$, $\tau$, and $(i_1,j_1,\ldots,i_p,j_p)\in I_p$. Then we have
\begin{enumerate}[$\bullet$]
\item $f_k(\xbar)=f_{\sigma(1)}(\xbar)\leq g_{\tau(1)}(\xbar)\leq\ldots\leq f_{\sigma(p)}(\xbar)\leq g_{\tau(p)}(\xbar)=g_k(\xbar)$,
\item for all $s\in [k]\backslash\Im\sigma$ and $t\in [k]\backslash\Im\tau$, neither $f_s(\xbar)$ nor $g_t(\xbar)$ is in $[f_k(\xbar),g_k(\xbar)]$, and
\item $y\in\bigcup_{t=1}^p[f_{\sigma(t)}(\xbar),g_{\tau(t)}(\xbar)]_m$.
\end{enumerate}
Therefore, it suffices to show that $A_{\xbar}\cap[f_k(\xbar),g_k(\xbar)]=\bigcup_{t=1}^p[f_{\sigma(t)}(\xbar),g_{\tau(t)}(\xbar)]_m$. Since $\xbar\in\pi(A)$, there are $1\leq p_*\leq k$, distinct $u_1,\ldots,u_{p_*}\in [k]$, and distinct $v_1,\ldots,v_{p_*}\in [k]$ satisfying $(\ast)_1$ through $(\ast)_3$ above. We want to show that $p=p_*$ and, for all $1\leq t\leq p$, $u_t=\sigma(t)$ and $v_t=\tau(t)$. First, 
$$
p=|\{1\leq t\leq k:f_t(\xbar)\in[f_k(\xbar),g_k(\xbar)]\}|=p_*.
$$
Next, we have
\begin{enumerate}[$\bullet$]
\item $f_k(\xbar)=f_{u_1}(\xbar)\lhd_m g_{v_1}(\xbar)<_m\ldots<_m f_{u_p}(\xbar)\lhd_m g_{u_p}(\xbar)=g_k(\xbar)$, and
\item $f_k(\xbar)=f_{\sigma(1)}(\xbar)\leq g_{\tau(1)}(\xbar)\leq\ldots\leq f_{\sigma(p)}(\xbar)\leq g_{\tau(p)}(\xbar)=g_k(\xbar)$.
\end{enumerate}
Therefore it must be the case that, for all $1\leq t\leq p$, we have $u_t=\sigma(t)$ and $v_t=\tau(t)$.

Part $(b)$. We continue to assume $s=t=k$. Suppose first that $g_k=f_k+c$ for some $c\geq 0$, and let $B$ be as in part $(a)$. Then $A$ is interdefinable with $A^*:=A\backslash B$. If $\xbar\in \pi(A_*)$ then, since $A$ has sorted fibers, there are $\sigma,\tau\in S_k$ such that
\begin{enumerate}[$\bullet$]
\item $f_{\sigma(1)}(\xbar)\lhd_m f_{\tau(1)}(\xbar)<_m\ldots<_m f_{\sigma(k)}(\xbar)\lhd_m g_{\tau(k)}(\xbar)$, and
\item $A_{\xbar}=\bigcup_{t=1}^k[f_{\sigma(t)}(\xbar),g_{\tau(t)}(\xbar)]_m$.
\end{enumerate}
Therefore, for some $1\leq p\leq k-1$, there are distinct $s_1,\ldots,s_p\in [k-1]$ and distinct $t_1,\ldots,t_p\in [k-1]$ such that
\begin{enumerate}[$\bullet$]
\item $f_{s_1}(\xbar)\lhd_m f_{t_1}(\xbar)<_m\ldots<_m f_{s_p}(\xbar)\lhd_m g_{t_p}(\xbar)$, and
\item $A^*_{\xbar}=\bigcup_{j=1}^p[f_{s_j}(\xbar),g_{t_j}(\xbar)]_m$.
\end{enumerate}
In particular, if $F=\{f_1,\ldots,f_{k-1}\}$ and $G=\{g_1,\ldots,g_{k-1}\}$, then $A^*$ has weakly sorted fibers witnessed by $(F,G,m)$. By Proposition \ref{prop:wsf}, $A_*$ is interdefinable with a finite sequence $(B_i)$ of Presburger sets, where each $B_i$ has sorted fibers witnessed by some $(\fbar^i,\gbar^i,m)$, with $|\fbar^i|=|\gbar^i|\leq k-1$.

Now, we must prove the other case: $g_k=f_k+c$ for some $c<0$. The argument is similar, and equally technical, so we sketch the proof and leave the details to the reader. First, set
$$
B=\{(\xbar,y)\not\in A:g_k(\xbar)\leq y\leq f_k(\xbar),~y\in m\Z\}.
$$
Then, using a similar argument as in part $(a)$, one shows that $B$ is definable in $(\Z,+,0)$, and so $A$ is interdefinable with $A^*=A\cup B$. Moreover, similar to the first case, $A^*$ has weakly sorted fibers, witnessed by $(F,G,m)$ where $F=\{f_1,\ldots,f_{k-1}\}$ and $G=\{g_1,\ldots,g_{k-1}\}$. The result then follows from Proposition \ref{prop:wsf} as in the first case.
\end{proof}

\section{Proof of the main result}\label{sec:proof}

In this section, we prove Theorem \ref{thm:main} by induction on $n$. The base case $n=1$ was done in Corollary \ref{cor:1dim}. So fix $n>0$ and, for a primary induction hypothesis, assume that if $B\seq\Z^n$ is a Presburger set then either $B$ is definable in $(\Z,+,0)$ or $B$ defines the ordering. 

\begin{Claim}\label{claim:main}
Fix a Presburger set $A\seq\Z^{n+1}$, which has sorted fibers, witnessed by $(\fbar,\gbar,m)$, where $\fbar=(f_1,\ldots,f_k)$ and $\gbar=(g_1,\ldots,g_k)$ are tuples of standard $\Z$-linear functions. Assume $X:=\pi(A)$ is a quasi-coset in $\Z^n$. Then there are $s,t\in\{1,\ldots,k\}$ such that $f_s-g_t$ is constant on $X$. 
\end{Claim}
\begin{proof}
Suppose not. Let $X=C\backslash Z$ where $C$ is a coset of rank $n_*\leq n$ and $Z\seq C$ is definable in $(\Z,+,0)$, with $\rk(Z)<n_*$. Let $C=\cbar+G$ where $G\leq\Z^n$ is a subgroup with basis $\alpha=\{\abar^1,\ldots,\abar^{n_*}\}$. Set $\Phi=\Phi_{\alpha,\cbar}|_X$, and let $W=\Phi(Z)$.  Then $\Phi\colon X\to\Z^{n_*}$ is an injective function with $\Im(\Phi)=\Z^{n_*}\backslash W$. By Corollary \ref{cor:rectcoset}, $\rk(W)<n_*$. Let $Y=\Z^{n_*}\backslash W$, and note that $\Phi\colon X\to Y$ is a bijection.

For $1\leq t\leq k$, define
$$
\tilde{f}_t=f_t\circ\Phi\inv\mand \tilde{g}_t=g_t\circ\Phi\inv.
$$
Note that $\tilde{f}_t$ and $\tilde{g}_t$ are functions from $Y$ to $\Z$. Since $A$ has sorted fibers it follows that for all $\xbar\in Y$, there are $\sigma,\tau\in S_k$ such that
$$
\tilde{f}_{\sigma(1)}(\xbar)\leq \tilde{g}_{\tau(1)}(\xbar)<\ldots< \tilde{f}_{\sigma(k)}(\xbar)\leq \tilde{g}_{\tau(k)}(\xbar).
$$

For any $1\leq t\leq k$, $\tilde{f}_t$ and $\tilde{g}_t$ are each a composition of a standard $\Z$-linear function and $\Phi\inv$, which is of the form $\zbar\mapsto \cbar+\sum_{i=1}^{n_*}z_i\abar^i$. Therefore $\tilde{f}_t$ and $\tilde{g}_t$ determine affine transformations from $\R^{n_*}$ to $\R$. For any $s,t\in\{1,\ldots,k\}$, $f_s-g_t$ is non-constant on $X$, and so $\tilde{f}_s=\tilde{g}_t$ is non-constant on $Y$. Given $\sigma,\tau\in S_k$ and $\bullet\in\{+,-\}$, define the polyhedron
$$
P^\bullet(\sigma,\tau)=\bigcap_{t=1}^k\skov{H}^\bullet(\tilde{g}_{\tau(t)}-\tilde{f}_{\sigma(t)})\cap\bigcap_{t=1}^{k-1}H^\bullet(\tilde{f}_{\sigma(t+1)}-\tilde{g}_{\tau(t)}).
$$
Then, altogether, we have $Y\seq\bigcup_{\sigma,\tau\in S_k}P^+(\sigma,\tau)$. By Lemma \ref{lem:polyint}, we may fix $\mu,\nu\in S_k$ such that $r(P^+(\mu,\nu))=\infty$. Then, by Lemma \ref{lem:polyinv}, $r(P^-(\mu,\nu))=\infty$. 

Fix $\xbar\in P^-(\mu,\nu)\cap Y$. Then
$$
\tilde{g}_{\nu(k)}(\xbar)\leq \tilde{f}_{\mu(k)}(\xbar)<\ldots<\tilde{g}_{\nu(1)}(\xbar)\leq \tilde{f}_{\mu(1)}(\xbar).
$$
On the other hand, there are $\sigma,\tau\in S_k$ such that
$$
\tilde{f}_{\sigma(1)}(\xbar)\leq \tilde{g}_{\tau(1)}(\xbar)<\ldots< \tilde{f}_{\sigma(k)}(\xbar)\leq \tilde{g}_{\tau(k)}(\xbar).
$$
Then $\tilde{g}_{\nu(k)}(\xbar)<\tilde{g}_t(\xbar)$ for all $t\neq \nu(k)$, and $\tilde{g}_{\tau(1)}(\xbar)<\tilde{g}_t(\xbar)$ for all $t\neq\tau(1)$, which means $\tau(1)=\nu(k)$. But then $\tilde{g}_{\nu(k)}(\xbar)\leq \tilde{f}_{\sigma(1)}(\xbar)\leq g_{\tau(1)}(\xbar)$, and so $\tilde{g}_{\nu(k)}(\xbar)=\tilde{f}_{\sigma(1)}(\xbar)$. Then $\tilde{f}_{\sigma(1)}(\xbar)<\tilde{f}_t(\xbar)$ for all $t\neq \mu(k)$, and so $\sigma(1)=\mu(k)$. Altogether, we have shown that $\tilde{f}_{\mu(k)}(\xbar)=\tilde{g}_{\nu(k)}(\xbar)$ for all $\xbar\in P^-(\mu,\nu)\cap Y$. Therefore $\tilde{f}_{\mu(k)}=\tilde{g}_{\nu(k)}$ on $\R^{n_*}$ since otherwise $P^-(\mu,\nu)\cap Y$ would be contained in the hyperplane $H(\tilde{f}_{\mu(k)}-\tilde{g}_{\nu(k)})$, contradicting Lemma \ref{lem:polyint}. 

Now, if $\xbar\in X$ then
$$
f_{\mu(k)}(\xbar)=\tilde{f}_{\mu(k)}(\Phi(\xbar))=\tilde{g}_{\nu(k)}(\Phi(\xbar))=g_{\nu(k)}(\xbar),
$$
and so $f_{\mu(k)}=g_{\nu(k)}$ on $X$, which contradicts our original assumption.
\end{proof}

Finally, we proceed with the induction step. Fix a Presburger subset $A\seq\Z^{n+1}$, and assume $A$ does not define the ordering. We want to show $A$ is definable in $(\Z,+,0)$. By Corollary \ref{cor:red3}, we may assume $A$ has sorted fibers witnessed by $(\fbar,\gbar,m)$, where $\fbar=(f_1,\ldots,f_k)$ and $\gbar=(g_1,\ldots,g_k)$ are tuples of standard $\Z$-linear functions. We prove, by induction on $k$, that $A$ is definable in $(\Z,+,0)$.

Suppose $k=1$ and let $X=\pi(A)$. Then $X$ is definable from $A$ and therefore is a Presburger subset of $\Z^n$, which does not define the ordering. By the primary induction hypothesis, $X$ is definable in $(\Z,+,0)$. Therefore, by Theorem \ref{thm:qc}, we may write
$$
X=X_1\cup\ldots\cup X_p,
$$
where each $X_i$ is a quasi-coset in $\Z^n$. For $1\leq i\leq p$, let $A_i=\{(\xbar,y)\in A:\xbar\in X_i\}$. Then $A=\bigcup_{i=1}^pA_i$ and so $A$ is interdefinable with $(A_i)_{1\leq i\leq p}$. Note also that each $A_i$ still has sorted fibers witnessed by $(\fbar,\gbar,m)$. Altogether, without loss of generality, we may assume $X=\pi(A)$ is a single quasi-coset in $\Z^n$.

Since $A$ has sorted fibers and $k=1$, it follows that, for all $\xbar\in X$, $f_1(\xbar)\leq g_1(\xbar)$ and $A_{\xbar}=[f_1(\xbar),g_1(\xbar)]_m$. Therefore, by Claim \ref{claim:main}, $g_1=f_1+c$ on $X$ for some $c\geq 0$. By Lemma \ref{lem:fibers}$(a)$, it follows that $A$ is definable in $(\Z,+,0)$. 

Now, for a secondary induction hypothesis, fix $k>1$ and assume that if $B\seq\Z^{n+1}$ is a Presburger set, which does not define the ordering and has sorted fibers witnessed by some $(\fbar',\gbar',m)$ such that $|\fbar'|=|\gbar'|<k$, then $B$ is definable in $(\Z,+,0)$. Let $A\seq\Z^{n+1}$ be a Presburger set, which does not define the ordering and has sorted fibers witnessed by $(\fbar,\gbar,m)$, where $\fbar=(f_1,\ldots,f_k)$ and $\gbar=(g_1,\ldots,g_k)$. 

As in the secondary base case, it suffices to assume $X=\pi(A)$ is a single quasi-coset in $\Z^n$. By Claim \ref{claim:main}, there are $s,t\in \{1,\ldots,k\}$ such that $f_s-g_t$ is constant on $X$. By Lemma \ref{lem:fibers}$(b)$ and the secondary induction hypothesis, it follows that $A$ is definable in $(\Z,+,0)$. This completes the proof of Theorem \ref{thm:main}.

\end{document}